\documentclass[12pt,a4paper]{amsart}
\usepackage{amsfonts, comment}
\usepackage{amsthm}
\usepackage{amsmath, amscd, amssymb, bm}
\usepackage{t1enc}
\usepackage[mathscr]{eucal}
\usepackage{indentfirst}
\usepackage{graphicx}
\usepackage{graphics, tikz}
\usepackage{hyperref}
\usepackage{pict2e}
\usepackage{epic}
\numberwithin{equation}{section}

\usepackage{epstopdf}

\usepackage{enumitem}
\usepackage{times}
\usepackage[english]{babel}
\usepackage[utf8]{inputenc}

\DeclareMathOperator{\reg}{reg}

\DeclareMathOperator{\projdim}{projdim}

\newtheoremstyle{plainbr}
  {}{}{\itshape}{}{\bfseries}{.}{.5em}
  {\thmname{#1}\thmnumber{ #2}\thmnote{ #3}}

\newtheoremstyle{definitionbr}
  {}{}{\normalfont}{}{\bfseries}{.}{.5em}
  {\thmname{#1}\thmnumber{ #2}\thmnote{ #3}}

\newtheoremstyle{remarkbr}
  {}{}{\normalfont}{}{\itshape}{.}{.5em}
  {\thmname{#1}\thmnumber{ #2}\thmnote{ #3}}

\theoremstyle{plainbr}
\newtheorem{theorem}{Theorem}[section]

\newtheorem{proposition}[theorem]{Proposition}
\newtheorem{corollary}[theorem]{Corollary}

\theoremstyle{remarkbr}
\newtheorem{remark}[theorem]{Remark}

\theoremstyle{definitionbr}
\newtheorem{example}[theorem]{Example}

\newtheorem{definition}[theorem]{Definition}
\newtheorem{qquestion}[theorem]{Question}
\newtheorem{cconjecture}[theorem]{Conjecture}

\begin{document}

\title[Realizing regularity bounds for binomial edge ideals]{Characterization of some graphs realizing regularity bounds for binomial edge ideals}
\author[N. Erey]{Nursel Erey$^{1}$}
\author[M. Ergen]{Muhammed Ergen$^{2}$}
\author[T. Hibi]{Takayuki Hibi$^{3}$}

\address{$^{1}$Gebze Technical University \\ Department of Mathematics \\
Gebze \\ 41400 Kocaeli \\ Turkey} 

\email{nurselerey@gtu.edu.tr}

\address{$^{2}$Gebze Technical University \\ Department of Mathematics \\
Gebze \\ 41400 Kocaeli \\ Turkey} 

\email{mergen@gtu.edu.tr}

\address{$^{3}$Department of Pure and Applied Mathematics, Graduate School of Information Science and Technology, Osaka University, Suita, Osaka 565--0871, Japan}
\email{hibi@math.sci.osaka-u.ac.jp}

\subjclass[2020]{05E40, 13D02, 05C75.}

\keywords{Binomial edge ideals, Castelnuovo-Mumford regularity, Intersection graphs}

\begin{abstract}

 In this paper, we characterize all graphs $G$ satisfying \[\reg(S/J_G)=\ell(G)=c(G)\] where $\ell(G)$ is the sum of the lengths of the longest induced paths in each connected component of $G$ and $c(G)$ is the number of the maximal cliques of $G$. We also characterize all connected graphs $G$ that satisfy \[\reg(S/J_G)=\ell(G)=|V(G)|-\omega(G)+1\] where $\omega(G)$ is the clique number of $G$. Moreover, we investigate the possible values of the regularity of $S/J_G$ within the intervals $[\ell(G), c(G)]$ and $[\ell(G), |V(G)|-\omega(G)+1]$. 

\end{abstract}

\maketitle


    \section*{Introduction}

    Let $G$ be a simple graph on the vertex set $V(G)=[n]=\{1,\ldots,n\}$ and edge set $E(G)$. Also, let $K$ be a field and let $S=K[x_1,\ldots,x_n,y_1,\ldots,y_n]$ be the polynomial ring over $K$ in $2n$ indeterminates. The \textit{binomial edge ideal} of $G$, denoted by $J_G$, is defined as the ideal of $S$ generated by the binomials $f_{ij}=x_i y_j - x_j y_i$ where $i<j$ and $\{i,j\} \in E(G)$. A binomial edge ideal is thus a natural generalization of the ideal of $2$-minors of a $2\times n$-matrix of indeterminates. Binomial edge ideals were introduced independently by Herzog et al. \cite{BEICIS} and by Ohtani \cite{GIGS2M}. 

    In recent years, many algebraic and homological properties of binomial edge ideals, such as their Gröbner bases, graded free resolutions, Hilbert functions and Cohen-Macaulayness have been investigated, and the relations between the algebraic properties of $J_G$ and the combinatorial properties of $G$ have been studied (see \cite{BEIWQGB, CMBEI, BEICIS, BEIOG, BEIWPR, LaC, LaC2, HFBEI}).

    In particular, the Castelnuovo-Mumford regularity of binomial edge ideals has emerged as one of the main topics of research on binomial edge ideals. Various lower and upper bounds for the regularity of $S/J_G$ have been found in terms of the combinatorial properties of $G$, and the regularity of $S/J_G$ has been explicitly determined for some specific graph classes $G$ (see \cite{ LBEI, RBEICBG, TCMRBEI, BEIR3, PCRBEI, RBEICG, RBBEI}). Matsuda and Murai \cite{RBBEI} demonstrated that \[\ell(G) \leq \reg(S/J_G) \leq n-1\] where $\ell(G)$ is the sum of the lengths of the longest induced paths in each connected component of $G$. In \cite{PCRBEI}, Malayeri et al. found a combinatorial upper bound for $\reg(S/J_G)$, given by \[ \reg(S/J_G) \leq c(G) \] where $c(G)$ is the number of maximal cliques in $G$. Another upper bound for $\reg(S/J_G)$ was obtained by Ene et al. in \cite{LBEI}, expressed as \[\reg(S/J_G) \leq n-\omega(G)+1\] where $\omega(G)$ is the clique number of $G$. The upper bounds $n-\omega(G)+1$ and $c(G)$ are incomparable; each can be sharper depending on the combinatorial properties of $G$.

    In this paper, we focus on the conditions under which the above mentioned lower and upper bounds for the regularity of binomial edge ideals coincide. The family of graphs we consider can be viewed as Cameron-Walker graph analogues of edge ideals (see \cite{ASCWG}). We base this analogy on the fact that Cameron-Walker graphs are exactly those graphs whose edge ideals' regularity realize the well-known upper and lower bounds \cite{MIEIHGBN,CIBNGI} in terms of (induced) matching numbers of the graph.

    In \cite{RBEICG}, Malayeri et al. characterized all chordal graphs satisfying the equality $\reg(S/J_G)=\ell(G)=c(G)$, and referred to them as strongly interval graphs. In the present paper, we extend their result to all graphs, regardless of the chordal condition. More precisely, we characterize in Corollary~\ref{CLgraphcorollary} all graphs $G$ satisfying $\reg(S/J_G)=\ell(G)=c(G)$ and we call them CL-graphs. Our next main result, Corollary~\ref{WLgraphcorollary} provides a complete characterization of all connected graphs $G$ satisfying the equality $\reg(S/J_G)=\ell(G)=|V(G)|-\omega(G)+1$, which we call WL-graphs.

    In Sections \ref{subsec:CL} and \ref{subsec:WL} we investigate the possible values of the regularity of $S/J_G$ within the intervals $[\ell(G), c(G)]$ and $[\ell(G), |V(G)|-\omega(G)+1]$ defined by the graph-dependent lower and upper bounds. We show in Theorem~\ref{2lrcthrm} that for any given values $2 \leq \ell \leq r \leq c$, there exists a graph $G$ such that $\ell = \ell(G)$, $r = \reg(S/J_G)$ and $c = c(G)$. Furthermore, we prove in Theorem~\ref{3lrwthrm} that for any prescribed values $3 \leq \ell \leq r \leq \overline{\omega}$, there exists a connected graph $G$ satisfying the equations $\ell = \ell(G)$, $r=\reg(S/J_G)$ and $\overline{\omega} = |V(G)| - \omega(G) +1$.
    

    \section{Preliminaries}

    \noindent \subsection{Graph Theory Background} Let $G$ be a finite simple graph with vertex set $V(G)$ and edge set $E(G)$. A sequence $P : v_0,v_1,\ldots,v_\ell$ is called a \textit{path} of length $\ell$ in $G$ if $v_0,v_1,\ldots,v_\ell$ are distinct vertices of $V(G)$ and $\{v_{i-1},v_{i}\} \in E(G)$ for all $1 \leq i \leq \ell$. A \textit{path graph} $P_{\ell+1}$ of length $\ell$ is the graph consisting of exactly such a sequence of $\ell+1$ vertices and $\ell$ edges. A \textit{cycle} $C$ of length $m$ in $G$ is a sequence $C : u_1,\ldots,u_m,u_1$ where $u_1,\ldots,u_m$ are distinct vertices of $G$ and $\{u_1,u_2\},\ldots,\{u_{m-1},u_m\},\{u_1,u_m\} \in E(G)$. A \textit{cycle graph} of length $m$ is usually denoted by $C_m$. A graph is said to be \textit{connected} if there is a path between every pair of its vertices; otherwise, it is said to be \textit{disconnected}.

    Let $G$ be a graph and $T \subseteq V(G)$. The \textit{induced subgraph} of $G$ on the vertex set $T$ is denoted by $G[T]$ and defined as the graph whose vertex set is $T$ and whose edge set consists of all edges in $E(G)$ that have both endpoints in $T$. A graph $G$ is called \textit{chordal} if $G$ does not contain $C_m$ as an induced subgraph for all $m \geq 4$. We denote the length of a longest induced path of a connected graph $G$ by $\ell(G)$. For any graph $G$, we write $\ell(G) = \ell (G _1) + \cdots + \ell (G_t)$ where $G_1,\ldots,G_t$ are the connected components of $G$.

    A \textit{complete graph} $K_n$ is a graph with $n$ vertices such that every pair of its vertices are adjacent. A \textit{clique} $F$ of a graph $G$ is a subset of $V(G)$ such that the induced subgraph $G[F]$ of $G$ is isomorphic to a complete graph. A \textit{maximal clique} is a clique that is not contained by another clique. A vertex $v$ is said to be a \textit{simplicial vertex} if it belongs to exactly one maximal clique.

    The \textit{union} of two graphs $G_1$ and $G_2$ is denoted by $G_1 \cup G_2$ and defined as the graph with the vertex set $V(G_1 \cup G_2) = V(G_1) \cup V(G_2)$ and edge set $E(G_1 \cup G_2) = E(G_1) \cup E(G_2)$. The \textit{disjoint union} of $G_1$ and $G_2$ is denoted by $G_1 \sqcup G_2$ and defined as the union of $G_1$ and $G_2$ with disjoint vertex sets.

    \begin{definition}[\cite{IGT}]
        Let $\mathcal{S} = \{S_1,\ldots,S_n\}$ be a collection of sets. The \textit{intersection graph} $G$ on $\mathcal{S}$ is defined as the graph with the set of vertices $V(G)=\{v_1,\ldots,v_n\}$ where each $v_i$ corresponds to the set $S_i$, and with the set of edges $E(G)=\{ \{v_i,v_j\} : i \neq j, \text{ } S_i \cap S_j \neq \emptyset \}$.
    \end{definition}

    Note that we can directly consider the vertices of the intersection graph as the family of sets that define it, that is, $V(G)=\{S_1,\ldots,S_n\}$ and two vertices of $G$ are adjacent if and only if their intersection is nonempty.
    
    \begin{remark}
        It is well-known that every graph can be expressed as an intersection graph. So, there are relations between a graph and the sets that define the intersection graph that represents it. In  this way, intersection graphs can be used to characterize some graph classes. For example, Gavril showed in \cite{TIGSTECG} that a graph is chordal if and only if it is the intersection graph of a collection of subtrees of a tree.
    \end{remark}
    
    \noindent \subsection{Algebra Background} Let $K$ be a field and $R=K[x_1,...,x_n]$. Let $I \subsetneq R$ be an ideal which can be generated by homogeneous elements of $R$. The \textit{minimal graded free resolution} of $R/I$ has the form
     $$
		\mathbb{F}: 0 \rightarrow \bigoplus_{j \in \mathbb{N}} R(-j)^{\beta_{p,j}} \stackrel{d_p}{\longrightarrow} \cdots \rightarrow \bigoplus_{j \in \mathbb{N}} R(-j)^{\beta_{1,j}} \stackrel{d_1}{\longrightarrow} R \stackrel{d_0}{\longrightarrow} R/I \rightarrow 0
    $$
    where $R(-j)$ is the polynomial ring obtained by shifting the degrees of $R$ by $j$.
    
    Note that, only a finite number of the numbers $\beta_{i,j}$ in $\mathbb{F}$ are nonzero. The numbers $\beta_{i,j}$ (or $\beta_{i,j}(R/I)$) are called the \textit{graded Betti numbers} of $R/I$. In addition, the \textit{projective dimension} of $R/I$ is defined as $$ \projdim(R/I) = \max\{ i : \beta_{i,j}(R/I) \neq 0 \text{ for some } j \}, $$ that is, the length $p$ of the minimal graded free resolution of $R/I$. Also, the \textit{Castelnuovo-Mumford regularity} (or simply \textit{regularity}) of $R/I$ is defined as
    $$\reg(R/I)=\max \{ j : \beta_{i,i+j}(R/I) \neq 0 \text{ for some } i \}.$$

    Various lower and upper bounds for the regularity of binomial edge ideals have been found. Also, the regularity of binomial edge ideals of some special graph classes have been calculated.
    
    The following theorem is an immediate consequence of \cite[Theorem 2.1]{BEIOG} which allows to compute the regularity of binomial edge ideals of complete graphs.

    \begin{theorem}\label{completereg}
        $\reg(S/J_{K_n}) = 1$ for any $n \geq 2$.
    \end{theorem}
    
    In \cite{RBBEI}, Matsuda and Murai found the following important results regarding the regularity of binomial edge ideals. They first determined the regularity of binomial edge ideals of path graphs as follows. 

    \begin{proposition}[\cite{RBBEI}] \label{pathreg}
         If $P$ is a path graph of length $\ell$, then $\reg(S/J_P) = \ell$.
    \end{proposition}

    In the same work, the authors gave a relation between the regularity of the binomial edge ideal of a graph and those of its induced subgraphs.

    \begin{theorem}[\cite{RBBEI}] \label{indsubgraphreg}
        Let $G$ be a graph and let $H$ be an induced subgraph of $G$. Then, $\reg(S/J_H) \leq \reg(S/J_G)$.
    \end{theorem}

    In addition to Propositions \ref{pathreg} and Theorem \ref{indsubgraphreg}, the authors of \cite{RBBEI} gave lower and upper bounds for the regularity of a binomial edge ideal in the following theorem.
    
    \begin{theorem}[{\cite[Theorem 1.1]{RBBEI}}] \label{lowerupperbreg} 
            Let $G$ be a graph on $[n]$. Then, $$\ell (G) \leq \reg(S/J_G) \leq n-1.$$
    \end{theorem}

    While Theorem \ref{lowerupperbreg} provides an upper bound for $\reg(S/J_G)$, another upper bound was later obtained by Malayeri et al. in \cite{PCRBEI} as follows.
    
    \begin{theorem}[{\cite[Corollary 2.8]{PCRBEI}}] \label{numberofmc}
            For any graph $G$, $\reg (S/J_G) \leq c(G)$ where $c(G)$ is the number of the maximal cliques of $G$.
    \end{theorem}

    As noted in \cite[Remark 1.3]{TSBTBEI}, for any disconnected graph, the regularity of its binomial edge ideal can be expressed in terms of its connected components as follows.

    \begin{remark} \label{regscc}
        Let $G$ be a graph, and let $G_1,\ldots,G_t$ be its connected components. Then, $$\reg(S/J_G)= \sum_{i=1}^t \reg(S_i/J_{G_i})$$ where $S_i=K[x_v,y_v: v \in V(G_i)]$ for all $1 \leq i \leq t$.
    \end{remark}
  
    \section{CL-Graphs}

    In this section, we characterize all graphs $G$ that satisfy $\ell(G)=c(G)$. The authors of \cite{RBEICG} investigated when the condition $\ell(G) = c(G)$ is satisfied for a chordal graph $G$. For this purpose, they introduced strongly interval graphs as follows.

    \begin{definition}[{\cite[Definition 4.1]{RBEICG}}] \label{strintgra}
        Let $\ell \in \mathbb{N}$ and $r \in \mathbb{N}_0 = \mathbb{N} \cup \{ 0 \}$. Also, let $J_0$, $J_1$, $\ldots$ , $J_\ell$, $I_1$, $\ldots$ , $I_r$ be subsets of $\mathbb{R}$ satisfying the following properties; 
        
        \begin{enumerate}[label=\roman*)]
		\item $J_0=[0] = \{ 0 \}$ and $J_j=[j-1,j]$ for $1 \leq j \leq \ell$.
  
		\item For each $1 \leq i \leq r$, $I_i=[a_i,b_i]$ where $a_i \in \mathbb{N}_0$ and $a_i<b_i< \ell$.
        
	\end{enumerate}
        The intersection graph on $\{J_0,J_1,\ldots,J_\ell,I_1,\ldots,I_r\}$ is called a \textit{connected strongly interval graph}. Also, if every connected component of a graph is isomorphic to a connected strongly interval graph, it is called a \textit{strongly interval graph}.
    \end{definition}

    \begin{remark}
        Strongly interval graphs constitute a subclass of interval graphs, and since all interval graphs are known to be chordal, strongly interval graphs are also chordal graphs.
    \end{remark}

    Based on Definition \ref{strintgra}, the following result was established in \cite{RBEICG}.
    
    \begin{theorem}[{\cite[Corollary 4.3]{RBEICG}}] \label{sigtheorem}
        Let $G$ be a chordal graph. Then, the following are equivalent; 
        \begin{enumerate}[label=(\alph*)]
            \item $\reg(S/J_G) = \ell (G) = c(G)$.
            \item $G$ is a strongly interval graph.
        \end{enumerate}
    \end{theorem}

    We now generalize the characterization in Theorem \ref{sigtheorem} by defining a broader class of strongly interval graphs and eliminating the need for the chordal condition. To this end, we introduce CL-graphs.

    \begin{definition} \label{CLgraphdef}
        Let $\ell \in \mathbb{N}$ and $r \in \mathbb{N}_0$. Also, let $J_0,J_1,\ldots,J_\ell,I_1,\ldots,I_r$ be subsets of $\mathbb{R}$ satisfying the following properties:
        \begin{enumerate}[label=\roman*)]
		\item $J_0=[0]$ and $J_j=[j-1,j]$ for $1 \leq j \leq \ell$.
  
		\item For each $1 \leq i \leq r$, there is some $t \in \mathbb{N}$ such that
        \[I_i=[a_1,b_1] \cup [a_2,b_2] \cup \cdots \cup       [a_t,b_t]\]
         for some
        $a_1,a_2,\ldots,a_t \in \mathbb{N}_0$ where
        \[a_1<b_1<a_2<b_2<\cdots<a_t<b_t<\ell\] with $a_{k+1} - b_k > 2$ for all $1 \leq k \leq t-1$.

            \item If $\{I_{i_k}\}_{k=1}^t$ is a subset of $\{I_1,\ldots,I_r \}$ such that $I_{i_p} \cap I_{i_q} \neq \emptyset$ for all $1 \leq p < q \leq t$, then $\bigcap_{k=1}^t I_{i_k} \neq \emptyset$.

            \item If $I_p \cap I_q \neq \emptyset$ and $j \in I_p\setminus I_q$ for some $1 \leq j \leq \ell$, then: \\
            $j+1 \notin I_p$ implies $j+1 \notin I_q$; \\
            $j-1 \notin I_p$ implies $j-1 \notin I_q$.
	\end{enumerate}
        We call the intersection graph on $\{J_0,J_1,\ldots,J_\ell,I_1,\ldots,I_r\}$ a \textit{connected CL-graph}. Also, if every connected component of a graph is isomorphic to a connected CL-graph, we call such a graph \textit{CL-graph}.
    \end{definition}

    \begin{remark}\label{r0path}
        Note that if $r=0$ in Definition \ref{CLgraphdef}, then the CL-graph is isomorphic to $P_{\ell+1}$ for all $\ell \geq 1$.
    \end{remark}

    \begin{example} \label{CLexample}
        Let $J_0=[0]$, $J_1 = [0,1]$, $J_2 = [1,2]$, $\ldots$ , $J_7 = [6,7]$. Also, let $I_1=[1,4.5]$, $I_2=[1,1.5] \cup [6,6.5]$ and $I_3=[3,3.5]$. One can easily check that the intersection graph $G$ on $\{J_0,J_1,J_2,\ldots,J_7,I_1,I_2,I_3\}$ displayed in Figure \ref{CLgraphexample} satisfies the conditions of being a CL-graph. Also, it can be seen that $\ell (G) = c(G)=7$.
        
        \begin{figure}[hbt!]
			\centering
			\begin{tikzpicture}[scale=0.9,auto=left,inner sep=0pt]
				\tikzstyle{every node}=[circle, minimum size=5pt,fill]
				\node (J0) at (0,0) [label=135:$J_0$]{};
				\node (J1) at (1.5,0) [label=135:$J_1$]{};
				\node (J2) at (3,0) [label=45:$J_2$]{};
				\node (J3) at (4.5,0) [label=225:$J_3$]{};
				\node (J4) at (6,0) [label=45:$J_4$]{};
				\node (J5) at (7.5,0) [label=45:$J_5$]{};
				\node (J6) at (9,0) [label=45:$J_6$]{};
				\node (J7) at (10.5,0) [label=45:$J_7$]{};
				\node (I1) at (3,3) [label=90:$I_1$]{};
				\node (I2) at (1.5,-3) [label=270:$I_2$]{};
				\node (I3) at (7.5,-3) [label=270:$I_3$]{};
				\foreach \from/\to in {J0/J1,J1/J2,J2/J3,J3/J4,J4/J5,J5/J6,J6/J7,I1/J1,I1/J2,I1/J3,I1/J4,I1/J5,I2/J1,I2/J2,I2/J6,I2/J7,I3/J3,I3/J4,I1/I2,I1/I3}
				\draw (\from) -- (\to);
			\end{tikzpicture}
                \caption{\label{CLgraphexample} CL-graph $G$ in Example \ref{CLexample}.}
                \end{figure}
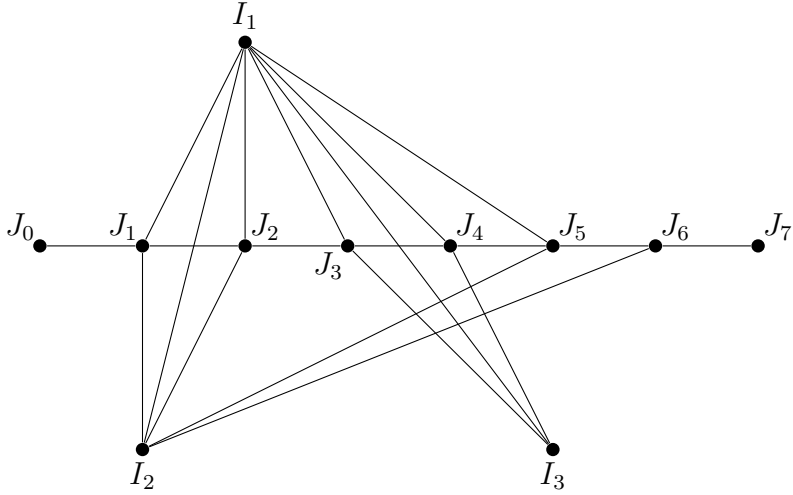
    \end{example}

    Before we proceed to the proof of the main theorem, we show that the class of strongly interval graphs is contained in the class of CL-graphs. We use a basic property of intervals in the real line, which corresponds to the one-dimensional case of the next theorem.

    \begin{theorem}[Helly's Theorem \cite{LODG}]
        Let $d \in \mathbb{N}$ and $\mathcal{C} = \{C_1,\ldots,C_n\}$ be a family of convex subsets of $\mathbb{R}^d$ where $n \geq d+1$. If the intersection of any $d+1$ element of $\mathcal{C}$ is nonempty, then $\bigcap_{i=1}^{n} C_i \neq 0$.
    \end{theorem}

    \begin{proposition}
        Every strongly interval graph is a CL-graph.
    \end{proposition}
    
    \begin{proof}
        We may examine only the connected case. Let $G$ be a strongly interval graph on $\left\{J_j\right\}_{j=0}^\ell \cup\left\{I_i\right\}_{i=1}^r$. Then, it clearly satisfies the conditions $i)$ and $ii)$ in Definition \ref{CLgraphdef}. Also, we know that each interval is a convex set in $\mathbb{R}$. Thus, Helly's Theorem implies that if $\{I_{i_k}\}_{k=1}^t$ is a subset of $\{I_1,\ldots,I_r \}$ with $I_{i_p} \cap I_{i_q} \neq \emptyset$ for all $1 \leq p < q \leq t$, then $\bigcap_{k=1}^t I_{i_k} \neq \emptyset$. So, $G$ satisfies the condition $iii)$ in Definition \ref{CLgraphdef}. Now, let $I_p=[a,b]$, $I_q=[c,d]$ with $a,c \in \mathbb{N}_0$. Suppose that $I_p \cap I_q \neq \emptyset$ and $j \in I_p\setminus I_q$ for some $j \in \mathbb{N}_0$. Let us consider these two cases:

        \vspace{3mm}
        \noindent \textit{Case 1:} Suppose that $j < c$. Then, we have $a < c \leq b$. If $j-1 \notin I_p$, then we see $j-1 < a < c$ and $j-1 \notin I_q$. If $j+1 \notin I_p$, then $j+1 > b$. Since $j,c \in \mathbb{N}_0$ and $j \notin I_q$, we see $j+1 \leq c$. Hence, we see $b < c$ and this is a contradiction.

        \vspace{3mm}
        \noindent \textit{Case 2:} Suppose that $j > d$. Then, we have $a \leq d < b$. If $j+1 \notin I_p$, then we see $j+1 > b > d$ and $j+1 \notin I_q$. If $j-1 \notin I_p$, then $j-1 < a$. Since $j,a \in \mathbb{N}_0$ and $j \in I_p$, we see $j = a$. Hence, we see $j \leq d$ and this is a contradiction.

        Thus, $G$ satisfies the condition $iv)$ in Definition \ref{CLgraphdef}. Therefore, $G$ is a CL-graph.
    \end{proof}

    \begin{remark}
        A CL-graph is not a strongly interval graph in general. Although every strongly interval graph is chordal, a CL-graph may not be chordal as in Example \ref{CLexample}.
    \end{remark}

    We are now ready to present the main theorem of this section.
    \begin{theorem} \label{CLgraphtheorem} Let $G$ be a graph. Then, the following are equivalent;
        \begin{enumerate}[label=(\alph*)]
            \item $\ell (G) = c(G)$.
            \item $G$ is a CL-graph.
        \end{enumerate}
    \end{theorem}

    \begin{proof} 
    It is sufficient to consider the case where $G$ is a connected graph. \\
    $\boldsymbol{(a \Rightarrow b)}$ : Suppose that $\ell (G) = c(G)$. Let $P : 0,1,2,\ldots,\ell$ be a longest induced path in $G$. For each $j \in [\ell]$, let $F_j$ be a maximal clique containing the edge $\{j-1,j\}$ of $P$. Since $P$ is an induced path, $F_1,\dots, F_\ell$ are all distinct. Since $c(G)=\ell$, it follows that $F_1,\ldots,F_\ell$ are all of the maximal cliques of $G$. Therefore, $F_j$ is the unique maximal clique containing the edge $\{j-1,j\}$. Let $W = V(G) \setminus V(P)$. If $W = \emptyset$, then $G=P$ and obviously, it is a CL-graph. So, we may continue with $W \neq \emptyset$. Suppose that $W = \{u_1, \ldots, u_r \}$. For any $u_i \in W$ and $j \in V(P)$, we have the following observations:

    \vspace{1mm}
    \begin{itemize}
        \item[$(\star)$]$u_i$ is adjacent to at least one vertex of the path $P$. Otherwise, $u_i$ belongs to a maximal clique not containing any vertices of $P$, which is not possible.
        
        \vspace{2mm}
        \item[$(\star \star)$] If $u_i$ is adjacent to $j$, then $u_i$ is adjacent to at least one of the vertices $j-1$ or $j+1$. Otherwise, $\{u_i,j\}$ is contained in a maximal clique other than $F_j$ and $F_{j+1}$. However, this is not possible since $P$ is an induced path.
    \end{itemize}
    
    For each $i \in [r]$, let $N_i = \{ j \in V(P) : \{j,u_i\} \in E(G) \}$. By $(\star)$ and $(\star \star)$, each $N_i$ is nonempty and the elements of $N_i$ can be written in increasing order as
    \begin{align}     
        & a_{i_1},a_{i_1}+1,\ldots,a_{i_1}+t_{i_1}, \nonumber \\ 
        & a_{i_2},a_{i_2}+1,\ldots,a_{i_2}+t_{i_2}, \nonumber \\ 
        &  \hspace{2.3cm} \vdots \nonumber \\ 
        & a_{i_{p_i}},a_{i_{p_i}}+1,\ldots,a_{i_{p_i}}+t_{i_{p_i}}. \nonumber
    \end{align}
    for some integers $0\leq a_{i_1}<a_{i_2}<\ldots<a_{i_{p_i}}\leq \ell-1$ and some positive integers $t_{i_1},\dots ,t_{i_{p_i}}$ satisfying $a_{i_{p_i}}+t_{i_{p_i}} \leq \ell$ and $a_{i_k}+t_{i_k}+2\leq a_{i_{k+1}}$ for each $k=1,\dots ,p_i-1$.  Let us construct a CL-graph. Let
    $$J_0=[0], \quad J_1=[0,1], \quad J_2=[1,2] \quad, \quad \ldots \quad, \quad J_\ell=[\ell-1,\ell]$$
    and for $i \in [r]$,
    {\small $$ I_i=\left[a_{i_1} \, , \, a_{i_1}+t_{i_1}-\frac{1}{2} \right] \cup \left[a_{i_2} \, , \, a_{i_2}+t_{i_2}-\frac{1}{2} \right] \cup \cdots \cup    \left[a_{i_{p_i}} \, , \, a_{i_{p_i}}+t_{i_{p_i}}-\frac{1}{2} \right].$$}
    
    \noindent Let $H$ be the intersection graph on $\{J_0,J_1,\ldots,J_\ell,I_1,\ldots,I_r \}$. Let us show that $H$ and $G$ are isomorphic graphs. Here, $J_j$'s correspond to the vertices $j$'s of $P$ and $I_i$'s correspond to $u_i$'s. \\
    Before we proceed, we note one more observation which will be useful.
     \begin{itemize}
        \item[$(\star\star\star)$] For any $j<\ell$ and $i\in [r]$, the vertex $u_i$ is adjacent to both $j$ and $j+1$ if and only if $j\in I_i$ by the definition of $I_i$. 
        \end{itemize}
    Keeping in mind that every edge of $P$ is contained in a unique maximal clique, it is straightforward using $(\star\star\star)$ to see that for any $i\neq k$, $I_i$ is adjacent to $I_k$ in $H$ if and only if $\{u_i,u_k\}\in E(G)$.

    \begin{itemize}
    \item Similarly, we have
    \begin{align*}
        \{I_i, J_j\}\in E(H)& \iff I_i \cap J_j \neq \emptyset\\
       &\iff I_i \cap \{j-1, j\}\neq \emptyset \\
       &\iff \{u_i, j\}\in E(G) \text{ by definition of } I_i.
    \end{align*}
   
  \item Lastly, for $j\neq k$, we have
\begin{align*}
\{J_j, J_k\}\in E(H)& \iff J_j \cap J_k \neq \emptyset\\
       &\iff \{j-1, j\} \cap \{k-1, k\} \neq \emptyset \\
       &\iff j=k-1 \text{ or } k=j-1\\
       &\iff \{j,k\} \in E(G).
\end{align*}
\end{itemize}

    Therefore, we conclude that $G$ and $H$ are isomorphic. Let us show that $H$ is a CL-graph.

    Since $H$ clearly satisfies the conditions $i)$ and $ii)$ of the definition of CL-graph, we first consider the condition $iii)$. Let $\mathcal{I}$ be a subset of $\{I_1,\ldots,I_r \}$ such that every pair of elements of $\mathcal{I}$ has non-empty intersection. Then, $\mathcal{I}$ is a clique in $H$. Let $U$ be the set of vertices in $W$ that correspond to $\mathcal{I}$. Then $U$ is a clique in $G$. Let $F_j$ be the maximal clique containing $U$. Then $\mathcal{I}\cup\{J_{j-1}, J_j\}$ is a clique in $H$. By $(\star\star\star)$ it follows that every element of $\mathcal{I}$ contains $j-1$, which shows that $iii)$ holds.
    
    Now, let us consider the condition $iv)$ of the definition of CL-graph. Suppose $I_p \cap I_q \neq \emptyset$, or equivalently, $\{u_p,u_q\}\in E(G)$. Also, suppose $j \in I_p \setminus I_q$ for some $j \in \mathbb{N}_0$. Then $u_p$ is adjacent to both $j$ and $j+1$. 

    \vspace{3mm}
    \noindent \textit{Case 1:} Suppose that $j+1 \notin I_p$. This implies $\{u_p,j+2\}\notin E(G)$. Assume for a contradiction $j+1 \in I_q$. Then $u_q$ is adjacent to both $j+1$ and $j+2$. On the other hand, since $j\notin I_q$, we know that $\{u_q,j\}\notin E(G)$. Recall that there are exactly two maximal cliques of $G$ that contains $j+1$, namely, $F_{j+1}$ and $F_{j+2}$. Then, the clique $\{u_p, u_q, j+1\}$ must be contained in either $F_{j+1}$ or $F_{j+2}$. However, each case gives a contradiction.

    \vspace{3mm}
    \noindent \textit{Case 2:} Assume for a contradiction that $j-1 \notin I_p$ but $j-1 \in I_q$. Similar to the proof of the previous case, one can deduce that $\{u_p,j-1\}\notin E(G)$ and $\{u_q,j+1\}\notin E(G)$. On the other hand, neither $F_j$ nor $F_{j+1}$ contains the clique $\{j, u_p, u_q\}$, a contradiction.

    Therefore, $H$ is a CL-graph, as desired.

    \vspace{3mm}
    \noindent $\boldsymbol{(b \Rightarrow a)}$ : Suppose that $G$ is a CL-graph on $\{J_0,J_1,\ldots,J_\ell,I_1,\ldots,I_r \}$. If $r=0$, then $G$ is a path of length $\ell$, as noted in Remark \ref{r0path}. In this case, we have $\ell(G)=c(G)$. So, we may continue with $r \geq 1$. Let $\mathcal{F} (G)$ be the set of all maximal cliques of $G$. Let us define $F_j=\{J_j,J_{j+1}\} \cup \{I_i : j \in I_i \}$ for each $0 \leq j \leq \ell - 1$. We claim that $\mathcal{F} (G) = \{F_0,F_1,\ldots,F_{\ell-1}\}$.
    
    Firstly, let us show $\{F_0,F_1,\ldots,F_{\ell-1}\} \subseteq \mathcal{F} (G)$. Clearly, $F_j$ is a clique in $G$ for each $0 \leq j \leq \ell -1 $. Let us show that $F_j$ is a maximal clique. Assume for a contradiction that there is a clique $F$ in $G$ such that $F_j \subsetneq F$. Then, there is an element $I_t \in F \setminus F_j$. Thus, $j \notin I_t$ by the construction of $F_j$. Since $F$ is a clique containing $\{J_j,J_{j+1},I_t\}$, we have $J_j \cap I_t \neq \emptyset$ and $J_{j+1} \cap I_t \neq \emptyset$. Here, $j \notin I_t$ and $J_j \cap I_t \neq \emptyset$ imply $[a,j-\epsilon] \subseteq I_t$ for some $a \leq j-1$ and $0 < \epsilon \leq 1$. Also, $j \notin I_t$ and $J_{j+1} \cap I_t \neq \emptyset$ imply $[j+1,b] \subseteq I_t$ for some $b > j+1$. Hence, $[a,j-\epsilon] \cup [j+1,b] \subseteq I_t$. However, $(j+1)-(j-\epsilon)=1+\epsilon \leq 2$ yields a contradiction to the definition of CL-graph.

    Now, we show $\mathcal{F} (G) \subseteq \{F_0,F_1,\ldots,F_{\ell-1}\}$. Assume the contrary. Then, there exists $F \in \mathcal{F}(G) \setminus \{F_0,F_1,\ldots,F_{\ell-1}\}$. If $F$ contains two elements of the form $J_j$ and $J_{j+1}$, then any element in $F$ contains $j$ and thus $F = F_j$, a contradiction. Therefore, we may assume that $F$ contains at most one element from the set $\{J_0,J_1,\ldots,J_{\ell}\}$.
    
    Suppose that $F \cap \{J_0,J_1,\ldots,J_\ell \} = \emptyset$. Then, $F\subseteq \{I_1,\ldots, I_r\}$. From the condition $iii)$ in Definition \ref{CLgraphdef}, we have $\bigcap_{f\in F} f \neq \emptyset$. Then, there exists $j \in \mathbb{N}_0$
    such that $j \in f$ for all $f\in F$ since the smallest element of every largest subinterval in $f$ is an integer. In this case, $F \subset F_j$, which is a contradiction.
    
    Now, suppose that $F \cap \{J_0,J_1,\ldots,J_\ell \} \neq \emptyset$. Then, $F=\{J_j\} \cup \{I_{i_k}\}_{k=1}^t$ for some $0 \leq j \leq \ell$ where $\{I_{i_k}\}_{k=1}^t \subseteq \{I_1,\ldots, I_r\}$. Since both $F$ and $F_j$ are maximal cliques, there is $p \in [t]$ such that $j \notin I_{i_p}$.
     Since $J_j \cap I_{i_p} \neq \emptyset$, and the smallest element of every largest subinterval in $I_{i_p}$ is an integer, we see $j-1 \in I_{i_p}$. Similarly, since both $F$ and $F_{j-1}$ are maximal cliques, there is $q \in [t]$ such that $j-1 \notin I_{i_q}$. The condition $iv)$ in Definition \ref{CLgraphdef} implies that $j \notin I_{i_q}$. Then, we see $J_j \cap I_{i_q} = \emptyset$ and this contradicts $F$ being a clique.

    Therefore, $\mathcal{F} (G) \subseteq \{F_0,F_1,\ldots,F_{\ell-1}\}$. Hence, we see $c(G)=\ell$. Moreover, $P:J_0,J_1,\ldots,J_{\ell}$ is an induced path of length $\ell$ in $G$. The inequality $\ell (G) \leq c(G)$ implies $\ell (G) = c(G)$.
    \end{proof}

    \begin{corollary} \label{CLgraphcorollary}
        Let $G$ be a graph. Then, the following are equivalent: 
        \begin{enumerate}[label=(\alph*)]
            \item $\reg(S/J_G) = \ell(G) = c(G)$.
            \item $G$ is a CL-graph.
        \end{enumerate}
    \end{corollary}

    \begin{proof}
        It follows from Theorems \ref{lowerupperbreg}, \ref{numberofmc}, and \ref{CLgraphtheorem}.
    \end{proof}
\subsection{Possible regularity values between $\ell(G)$ and $c(G)$}\label{subsec:CL}
    It is known that $\ell(G) \leq \reg(S/J_G) \leq c(G)$ for any graph $G$. This naturally leads to the question of for which triples $(\ell, r, c)$ of natural numbers with $\ell \leq r \leq c$ there exists a graph $G$ that satisfies the equations $\ell(G) = \ell$, $\reg(S/J_G) = r$ and $c(G) = c$.
    
    Let us consider the case $\ell=1$. Then, $\ell(G)=1$ implies that $G$ has the form $G = K_t \sqcup \left( \bigsqcup_{i=1}^{c-1} K_1 \right)$ for some $t \geq 2$ and $c \geq 1$. Then, we see $\ell(G)=1$, $\reg(S/J_G)=1$, $c(G)=c$ and $G$ is disconnected if $c \geq 2$. So, we conclude that there is no graph $G$ satisfying the desired equalities if $\ell=1$, and $2 \leq r \leq c$. We show in Theorem \ref{2lrcthrm} that there is a connected graph $G$ that satisfies the desired equalities for any $2 \leq \ell \leq r \leq c$. Before proceeding to the proof, we recall a well-known inequality involving the regularity of binomial edge ideals, which has been frequently used in the literature.

    Let $G$ be a graph on $[n]$ and let $v$ be a non-simplicial vertex in $G$. Then, \cite[Lemma 4.8]{GIGS2M} implies that $J_G = \left( J_{G-v} + (x_v,y_v)\right) \cap J_{G_v}$ where $G-v$ is the induced subgraph of $G$ on $V(G) \setminus \{v\}$ and $G_v$ is the graph obtained by $G$ by adding edges between the nonadjacent neighbours of $v$ in $G$, so that $v$ becomes a simplicial vertex. Then, there is an exact sequence
    \begin{align}
        0 \longrightarrow \frac{S}{J_G} \longrightarrow \frac{S_v}{J_{G-v}} \oplus \frac{S}{J_{G_v}} \longrightarrow \frac{S_v}{J_{G_v - v}} \longrightarrow 0 \tag{1} \label{exactseq}
    \end{align}
    where $S_v=K[x_i,y_i : i \in [n]\setminus \{v\} ]$. Applying \cite[Corollary 18.7]{GS} to the exact sequence (\ref{exactseq}) the following well-known inequality is obtained:
    \begin{align}
        \reg(S/J_G) \leq \max\{ \reg(S_v/J_{G-v}), \reg(S/J_{G_v}), \reg(S_v/J_{G_v-v})+1 \}. \tag{2} \label{regmaxineq}
    \end{align}
    
    We are now set to prove the following theorem.

    \begin{theorem} \label{2lrcthrm}
        For any positive integers $\ell,r,c$ with $2 \leq \ell \leq r \leq c$, there exists a connected graph $G$, with $\ell = \ell(G)$, $r=\reg (S/J_G)$ and $c = c(G)$.
     \end{theorem}
    \begin{proof}
        \textit{Case 1:} 
        Suppose that $\ell = 2$. Let us consider the graph $G$ with the set of vertices $V(G)=\{v,v_1,\ldots,v_r,v_1',\ldots,v_r',w_1,\ldots,w_{c-r}\}$ and the set of edges \[E(G) = \{ \{v,v_i\},\{v,v_i'\},\{v_i,v_i'\} : 1 \leq i \leq r \} \cup \{ \{v,w_j\} : 1 \leq j \leq c-r \}.\]       

        Clearly, $\ell (G) = 2$ and $c(G)=c$. We show that $\reg(S/J_G)=r$. We have $r = \reg(S/J_{G-v}) \leq \reg(S/J_G)$ by Proposition \ref{pathreg}, Theorem \ref{indsubgraphreg} and Remark \ref{regscc}. Also, $v$ is not a simplicial vertex in $G$. Then, we use the inequality (\ref{regmaxineq}), that is,
        $$
            \reg(S/J_G) \leq \max\{ \reg(S_v/J_{G-v}), \reg(S/J_{G_v}), \reg(S_v/J_{G_v-v})+1 \}.
        $$
        As we mentioned above, $\reg(S_v/J_{G-v})=r$. Also, $G_v$ and $G_v-v$ are complete graphs. So, $\reg(S/J_{G_v})=\reg(S_v/J_{G_v-v})=1$ from Theorem \ref{completereg}. Thus, we have the inequality
        $$
            \reg(S/J_G) \leq \max\{ r, 1, 2 \}.
        $$
        Since $r \geq \ell = 2$, it can be seen that $\reg(S/J_G) \leq r$. Hence, we see $\reg(S/J_G)=r$.

        \vspace{3mm}
        \noindent \textit{Case 2:} Suppose that $\ell \geq 3$. Let $H$ be the graph with the set of vertices $V(H)=\{v,v_1,\ldots,v_{r-\ell+2},v_1',\ldots,v_{r-\ell+2}',w_1,\ldots,w_{c-r}\}$ and the set of edges
        \begin{align*}
            E(H) = & \{ \{v,v_i\},\{v,v_i'\},\{v_i,v_i'\} : 1 \leq i \leq r-\ell+2 \} \\ & \cup \{ \{v,w_j\} : 1 \leq j \leq c-r \}.
        \end{align*}
        Also, let $P$ be the path graph with consecutive vertices $v_1,u_1,u_2,\ldots,u_{\ell-2}$ where $V(P) \cap V(H)=\{v_1\}$. Now, let us consider the graph $G=H \cup P$ displayed in Figure \ref{lg2 rc}.

        \begin{figure}[hbt!]
        \centering
	\begin{tikzpicture}[scale=0.9]
        \tikzset{
        place/.style={circle, minimum size=2mm,
        inner sep=0pt, outer sep=0pt, fill
        }}
        \node[place] (v) at (0,0) {};
        \node at (0.1,-0.32) {$v$};
        \node[place, label=270:$v_1$] (v1) at (-2,0) {};
        \node[place, label=125:$v'_1$] (vv1) at (-1.732,1) {};
        \node[place, label=135:$v_2$] (v2) at (-1,1.732) {};
        \node[place, label=90:$v'_2$] (vv2) at (0,2) {};
        \node at (0.8,1.4) {\rotatebox{148}{$\cdots$}};
        \node[place, label=45:$v_{r-\ell+2}$] (vr-l+2) at (1.732,1) {};
        \node[place, label=0:$v'_{r-\ell+2}$] (vvr-l+2) at (2,0) {};
        
        \node[place, label=225:$w_1$] (w1) at (-1.62,-1.18) {};
        \node[place, label=270:$w_2$] (w2) at (-0.62,-1.90) {};
        \node at (0.5,-1.5) {\rotatebox{20}{$\cdots$}};
        \node[place, label=315:$w_{c-r}$] (wcr) at (1.62,-1.18) {};

        \node[place, label=270:$u_1$] (u1) at (-4,0) {};
        \node[place, label=270:$u_2$] (u2) at (-6,0) {};
        \node at (-7,0) {\rotatebox{0}{$\cdots$}};
        \node[place, label=270:$u_{\ell-3}$] (ul-3) at (-8,0) {};
        \node[place, label=270:$u_{\ell-2}$] (ul-2) at (-10,0) {};
        
        \draw (v) -- (v1);
        \draw (v) -- (vv1);
        \draw (v) -- (v2);
        \draw (v) -- (vv2);
        \draw (v) -- (vr-l+2);
        \draw (v) -- (vvr-l+2);
        \draw (v1) -- (vv1);
        \draw (v2) -- (vv2);
        \draw (vr-l+2) -- (vvr-l+2);
        \draw (v) -- (w1);
        \draw (v) -- (w2);
        \draw (v) -- (wcr);
        \draw (v1) -- (u1);
        \draw (u1) -- (u2);
        \draw (ul-3) -- (ul-2);
            
	\end{tikzpicture}
        \caption{\label{lg2 rc} The graph $G = H \cup P$}
        \end{figure}
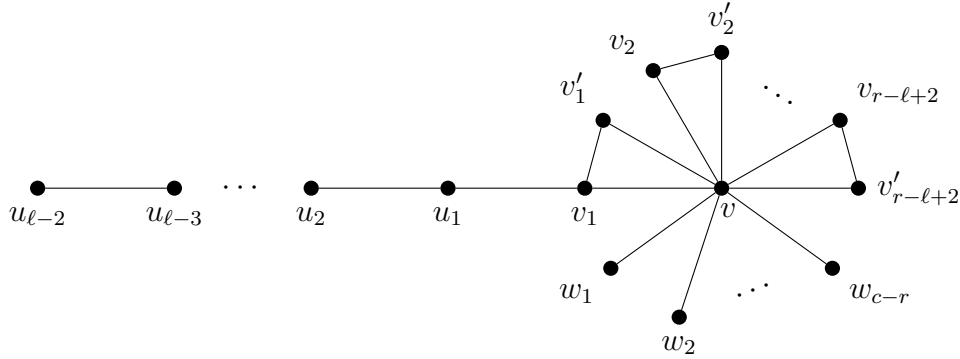

        Obviously, $\ell(G) = \ell$ and $c(G)=c$. We show that $\reg(S/J_G)=r$. We have $r = \reg(S/J_{G-v}) \leq \reg(S/J_{G})$ from Proposition \ref{pathreg}, Theorem \ref{indsubgraphreg} and Remark \ref{regscc}. On the other hand, $v$ is not a simplicial vertex in $G$. Then, we write the inequality
        $$
            \reg(S/J_G) \leq \max\{ \reg(S_v/J_{G-v}), \reg(S/J_{G_v}), \reg(S_v/J_{G_v-v})+1 \}.
        $$
        As stated above, $\reg(S_v/J_{G-v})=r$. Also, $\ell(G_v)= \ell -1$ and $c(G_v) = \ell -1$. By Theorem \ref{lowerupperbreg} and Theorem \ref{numberofmc}, we see $\reg(S/J_{G_v})=\ell-1$. Similarly, we see $\ell(G_v-v)= \ell -1$ and $c(G_v-v) = \ell -1$. From Theorem \ref{lowerupperbreg} and Theorem \ref{numberofmc}, we observe $\reg(S_v/J_{G_v-v})=\ell-1$. Then, we have the inequality
        $$
            \reg(S/J_G) \leq \max \{r, \ell-1, \ell \} \leq r.
        $$
        which proves that $\reg(S/J_G)=r$.
    \end{proof}

    \section{WL-Graphs}

    For any graph $G$, the \textit{clique number} of $G$ is denoted by $\omega(G)$, and is defined as the maximum size of a clique in $G$. The authors of \cite{LBEI} established an upper bound for $\reg(S/J_G)$ involving the clique number of $G$ as follows.

    \begin{theorem}[{\cite[Theorem 2.1]{LBEI}}] \label{wGreg}
        Let $G$ be a connected graph on $[n]$. Then, $$\reg(S/J_G) \leq n-\omega(G)+1.$$
    \end{theorem}

    For any connected graph $G$ on $[n]$, we know from Theorem \ref{lowerupperbreg} and Theorem \ref{wGreg} that $\ell(G) \leq \reg(S/J_G) \leq n-\omega(G)+1$. In this section, we investigate the class of connected graphs $G$ for which $\ell(G)=n-\omega(G)+1$ holds. For this purpose, we define WL-graphs as follows.

    \begin{definition} \label{WLdef}
         A graph $G$ is called a \textit{WL-graph} if there exist some positive integers $\ell, \omega$ with $\omega \geq 2$ such that $G$ can be written as $G=P \cup K \cup H$ where $P$, $K$ and $H$ satisfy the following properties:
        \begin{enumerate}[label=\roman*)]
            \item $P$ is a path graph with consecutive vertices $v_0,v_1,\ldots,v_\ell$.
            
            \item $K$ is a complete graph on the vertex set $\{v_{t},v_{t+1},u_1,\ldots,u_{\omega-2}\}$ for some $0 \leq t < \ell$ where $V(P) \cap V(K) = \{v_t,v_{t+1} \}$.

            \item $H$ is a graph with the vertex set $V(H)=V(P) \cup V(K)$ and edge set $E(H) \subseteq \{ \{v_i,u_j\}: 0 \leq i \leq \ell, \hspace{3mm} 1 \leq j \leq \omega-2 \}$.
        \end{enumerate}

    \end{definition}

    Note that if $\omega=2$ in Definition \ref{WLdef}, then $G=P_{\ell+1}$. Next, we characterize WL-graphs by the equality $\ell (G) = n - \omega(G) +1$.
    
    \begin{theorem} \label{WLregthm}
        Let $G$ be a connected graph on $[n]$. Then, the following are equivalent:
        \begin{enumerate}[label=(\alph*)]
            \item $\ell (G) = n - \omega(G) +1$.
            \item $G$ is a WL-graph.
        \end{enumerate}
    \end{theorem}
    \begin{proof}
        $\boldsymbol{(b \Rightarrow a)}$ : Let $G$ be a WL-graph with the path $P:v_0,v_1, \ldots, v_\ell$ and the complete graph $K$ with $V(K)=\{v_t,v_{t+1},u_1,\ldots,u_{\omega-2}\}$. We show $\ell(G)=\ell$ and $\omega(G)=\omega$.

        Firstly, let us show $\ell(G)=\ell$. Obviously, $P$ is an induced path of length $\ell$ in $G$. Hence, $\ell(G) \geq \ell$. Suppose that $Q$ is an induced path in $G$. Then, $V(Q)$ may contain at most two vertices from the set $V(K)$, otherwise this would contradict $Q$ being an induced path. Then, we have $|V(Q)| \leq |V(P) \setminus \{v_t,v_{t+1}\} |+2$. Hence, $|V(Q)| \leq \ell+1$ and $Q$ has length at most $\ell$. Therefore, we conclude $\ell(G)=\ell$.

        Now, let us show $\omega(G)=\omega$. Clearly, $V(K)$ is a clique with $|V(K)|=\omega$. Suppose that $F$ is a clique in $G$. Since $P$ is an induced path in $G$, we see that $F$ may contain at most two vertices from the set $V(P)$. Then, we have $|F| \leq |V(K) \setminus \{v_t,v_{t+1}\} |+2$ and hence, $|F| \leq \omega$. So, we conclude $\omega(G)=\omega$.

        As a last step, we write $|V(G)|=|V(P)|+|V(K) \setminus \{v_t,v_{t+1}\}|$ and hence, $n=\ell+\omega-1$. If we substitute $\ell = \ell(G)$ and $\omega=\omega(G)$, we obtain $n=\ell(G)+\omega(G)-1$. Therefore, $\ell(G)=n-\omega(G)+1$.
        
        \vspace{3mm}
        $\boldsymbol{(a \Rightarrow b)}$ : Suppose that $\ell (G) = n - \omega(G) +1$. Let $P:v_0,v_1,\ldots,v_{\ell}$ be a longest induced path in $G$ and let $F$ be a clique of maximum size in $G$ with $|F| = \omega$. Let $|V(P) \cap F|=m$.

        If $m>2$, then there exist at least three vertices $v_i,v_j,v_k \in F$. Hence, ${v_i,v_j,v_k}$ form a clique in $G$ but this contradicts $P$ being an induced path.

        If $m<2$, then $|V(P) \cup F| \leq |V(G)|$ implies $\ell(G)+1+\omega(G)-m \leq n$. Thus, $\ell(G) \leq n- \omega(G)+m-1$. Since $m < 2$, this contradicts the fact that $\ell (G) = n - \omega(G) +1$.

        The only possibility is $m=|V(P) \cap F|=2$. Hence, there are two vertices $v_t,v_s \in F$ with $0 \leq t < s \leq \ell$. Since $P$ is an induced path and $\{v_t,v_s\} \in E(G)$, then we have $s=t+1$. It means that $F$ contains two consecutive vertices $v_t,v_{t+1}$ of the path $P$. Then, we see that $F$ has the form $F = \{v_{t},v_{t+1},u_1,\ldots,u_{\omega-2}\}$ where $u_1,\ldots,u_{\omega-2} \in V(G) \setminus V(P)$. Since $\ell(G)=n-\omega(G)+1$, we write $|V(G)|=n=\ell+\omega-1=|V(P) \cup F|$. So, we see $V(G)=V(P) \cup F$. Let $K$ be the complete graph on the vertex set $V(K)=F$. Also, we consider the subgraph $H$ of $G$ with $V(H)=V(G)$ and $E(H)=\{ \{v_i,u_j\} \in E(G) :  0 \leq i \leq \ell, \hspace{1mm} 1 \leq j \leq \omega-2 \}$. Hence, we observe that $G=P \cup K \cup H$. Therefore, $G$ is a WL-graph.
    \end{proof}

    \begin{corollary} \label{WLgraphcorollary}
        Let $G$ be a connected graph on $[n]$. Then, the following statements are equivalent:
        \begin{enumerate}[label=(\alph*)]
            \item $\reg(S/J_G) = \ell(G) = n-\omega(G)+1$.
            \item $G$ is a WL-graph.
        \end{enumerate}
    \end{corollary}

    \begin{proof}
        It follows from Theorems \ref{lowerupperbreg}, \ref{wGreg}, and \ref{WLregthm}.
    \end{proof}

    Prompted by a question from an anonymous referee, we investigated whether the class of WL-graphs introduced in Definition \ref{WLdef} provides counterexamples to a conjecture proposed by Ficarra and Sgroi in \cite{TSBTBEI}.
    
    In order to state the conjecture, we recall that, a vertex $v$ of a graph $G$ is said to be \textit{isolated} if it is not adjacent to any other vertex of $G$; otherwise, it is said to be \textit{non-isolated}.
    
    \begin{cconjecture}[{\cite[Conjecture 3.3]{TSBTBEI}}] \label{regpdimconj}
        Suppose that $G$ is a graph on $n \geq 7$ non-isolated vertices. If $\reg(S/J_G)=n-2$, then $\projdim(S/J_G) \leq n+1$.
    \end{cconjecture}

    It was noted in \cite{TSBTBEI} that the above conjecture had already been verified for $n=7,8,9$. The following example with $n=10$ constitutes a counterexample to the conjecture within the class of WL-graphs.

    \begin{example} \label{ContraExample}
         Let $P$ be the path graph with consecutive vertices $v_0,v_1,\ldots,v_8$ and let $K$ be the complete graph on $V(K)=\{v_2,v_3,u\}$. In addition, let $H$ be the graph whose vertex set is $V(H)=\{v_0,v_1,\ldots,v_8,u\}$ and whose edge set is $E(H)=\{\{v_0,u\},\{v_5,u\},\{v_6,u\},\{v_8,u\}\}$. Then, the graph $G=P \cup K \cup H$ shown in Figure \ref{ContraExampleFigure} with $n=10$ non-isolated vertices is a WL-graph since it satisfies $\ell(G)=8=n-\omega(G)+1$.
         By using the computation tool Macaulay2 (see \cite{Mac2}), we find $\reg(S/J_G) = 8 = n-2$ but $\projdim(S/J_G) = 12 > n+1$. Thus, $G$ provides a counterexample to Conjecture \ref{regpdimconj}. 
    \end{example}
        \begin{figure}[hbt!]
			\centering
			\begin{tikzpicture}[scale=0.9,auto=left,inner sep=0pt]
				\tikzstyle{every node}=[circle, minimum size=5pt,fill]
				\node (v0) at (0,0) [label=270:$v_0$]{};
				\node (v1) at (1.5,0) [label=270:$v_1$]{};
				\node (v2) at (3,0) [label=270:$v_2$]{};
				\node (v3) at (4.5,0) [label=270:$v_3$]{};
				\node (v4) at (6,0) [label=270:$v_4$]{};
				\node (v5) at (7.5,0) [label=270:$v_5$]{};
				\node (v6) at (9,0) [label=270:$v_6$]{};
				\node (v7) at (10.5,0) [label=270:$v_7$]{};
                \node (v8) at (12,0) [label=270:$v_8$]{};
                
				\node (u) at (3.75,2.4) [label=90:$u$]{};
				\foreach \from/\to in {v0/v1,v1/v2,v2/v3,v3/v4,v4/v5,v5/v6,v6/v7,v7/v8,v2/u,v3/u,v0/u,v5/u,v6/u,v8/u}
				\draw (\from) -- (\to);
			    \end{tikzpicture}
                \caption{\label{ContraExampleFigure} WL-graph $G$ in Example \ref{ContraExample}.}
                \end{figure}
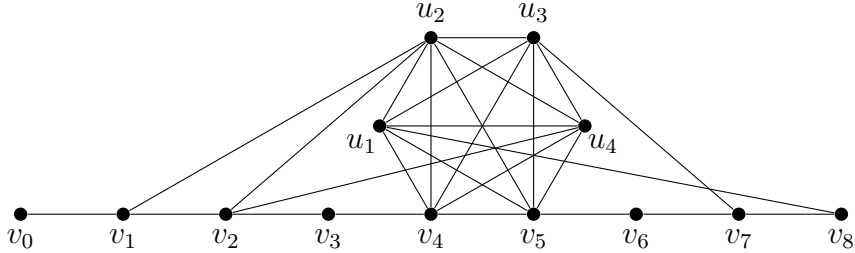
       
\subsection{Possible regularity values between $\ell(G)$ and $|V(G)|-\omega(G)+1$}\label{subsec:WL}
    It is known that $\ell(G) \leq \reg(S/J_G) \leq |V(G)|-\omega(G)+1$ for any connected graph $G$. It is natural to ask for which triples $(\ell, r, \overline{\omega})$ of natural numbers with $\ell \leq r \leq \overline{\omega}$ there exists a connected graph $G$ satisfying $\ell(G) = \ell$, $\reg(S/J_G) = r$ and $|V(G)|-\omega(G)+1 = \overline{\omega}$. We prove in Theorem \ref{3lrwthrm} that such connected graphs exist for all $(\ell,r,\overline{\omega})$ satisfying $3 \leq \ell \leq r \leq   \overline{\omega}$. Before presenting the proof, we recall a few basic concepts that will be needed in what follows.
    
    A vertex $v$ in a graph $G$ is called a \textit{cut vertex} if $G-v$ has more connected components than $G$. 
    
    Let $v$ be a cut vertex of a graph $G$ and let $H_1, \ldots, H_t$ be the connected components of $G - v$. Also, let $G_i$ be the induced subgraph of $G$ on the vertex set $V\left(H_i\right) \cup\{v\}$ for $1 \leq i \leq t$. Then, $G_1, \ldots, G_t$ are called the \textit{splits} of $G$ at $v$ \cite{RBEICBG}.

    \begin{theorem}[{\cite[Theorem 3.1]{RBEICBG}}] \label{gluingreg}
        Let $G$ be a graph and let $G_1, G_2$ be the splits of $G$ at a vertex $v \in V(G)$. If $v$ is a simplicial vertex in both $G_1$ and $G_2$, then
        $$
          \reg (S / J_G )= \reg (S / J_{G_1})+ \reg (S / J_{G_2}).
        $$
    \end{theorem}

    We are now in a position to present the proof.
    
    \begin{theorem} \label{3lrwthrm}
        Suppose that $\ell,r,\overline{\omega} \in \mathbb{N}$ with $3 \leq \ell \leq r \leq \overline{\omega}$. Then, there exists a connected graph $G$ such that $\ell = \ell(G)$, $r=\reg(S/J_G)$ and $\overline{\omega} = |V(G)| - \omega(G) +1$.
    \end{theorem}
    \begin{proof}       
        Let $U=\{u_1,\ldots,u_{\overline{\omega}-r}\}$, $V=\{v_1,\ldots,v_{\overline{\omega}-r}\}$, $Q=\{q_1,\ldots,q_{r-\ell}\}$, $Q'=\{q_1',\ldots,q_{r-\ell}'\}$, $\{1,2, \dots, \ell+1\}$ be disjoint sets. Let $K$ and $K'$ be the complete graphs on the vertex sets $\{1,2\} \cup U$ and $\{2,3\} \cup V \cup Q$, respectively. Moreover, let $P$ be the path graph with the consecutive vertices $1,2,\dots, \ell+1$ and $H$ be the graph consisting of the edges $\{q_k,q_k'\}$ for $1 \leq k \leq r-\ell$. Now, let us consider the graph $G = P \cup K \cup K' \cup H$.

        So, we have $\ell (G) = \ell$ and $|V(G)|-\omega(G)+1 = \overline{\omega}$. By applying Theorem \ref{gluingreg} to the vertex $2$, we obtain 
        $$
        \reg(S/J_G)=\reg(S/J_K)+\reg(S/J_{K' \cup H \cup (P-\{1,2\})}).
        $$ 
        By applying Theorem \ref{gluingreg} to the vertex $3$ in the graph $K' \cup H \cup (P-\{1,2\})$, we obtain 
        $$
        \reg(S/J_G)=\reg(S/J_K) + \reg(S/J_{K' \cup H}) +  \reg(S/J_{P-\{1,2\}}).
        $$ 
        Similarly, by applying Theorem \ref{gluingreg} to the vertices $q_1,q_2,\ldots,q_{r-\ell}$ in the graph $K' \cup H$ repeatedly, it can be seen that 
        $$
        \reg(S/J_G)=\reg(S/J_K)+\reg(S/J_{K'})+\reg(S/J_H)+\reg(S/J_{P-\{1,2\}}).
        $$ 
        As $K$ and $K'$ are complete graphs, we have $\reg(S/J_K)=\reg(S/J_{K'})=1$ by Theorem \ref{completereg}. Since $P-\{1,2\}$ is a path of length $\ell-2$ and $H$ is a disjoint union of $r-\ell$ disjoint paths of length $1$, Proposition \ref{pathreg} and Remark \ref{regscc} imply $\reg(S/J_{P-\{1,2\}})=\ell-2$ and $\reg(S/J_H)=r-\ell$. Therefore, it can be seen that $\reg (S/J_G) = r$.
    \end{proof}

    Note that Theorem \ref{3lrwthrm} cannot be extended to the case $\ell = 2$. For the case $\ell=2$, a connected graph $G$ satisfying the prescribed equalities does not always exist as we show in the next proposition.

    \begin{proposition}
        There is no connected graph $G$ such that      
        \[\ell(G)=2, \reg(S/J_G)=3 \text{ and } |V(G)| - \omega(G) +1=3.\]
    \end{proposition}
     
    \begin{proof}
    
    Assume that there exists a connected graph $G$ for $\ell=\ell(G)=2$, $r=\reg(S/J_G)=3$ and $\overline{\omega}=|V(G)| - \omega(G) +1=3$. Since $\overline{\omega}=3$, $G$ has the form $V(G)=\{v,u,1,2,\ldots,\omega \}$ where $F=\{1,2,\ldots,\omega \}$ is a clique in $G$ satisfying $\omega(G)=\omega$. We may examine two cases.

    \vspace{3mm}
    \noindent \textit{Case 1:} Suppose that $\{v,u\} \in E(G)$. Then, at least one of $v$ or $u$ has some neighbours from $F$. Without loss of generality, we may assume that the neighbours of $v$ belonging to $F$ are $1,2,\ldots,t$ where $t < \omega$ due to $\overline{\omega}=3$. If $u$ is not adjacent to at least one of the vertices $1,2,\ldots,t$, then $u$ is adjacent to all vertices $t+1,t+2, \ldots, \omega$ because of $\ell =2$. Also, $v$ is not a simplicial vertex in $G$. Then, we use the inequality (\ref{regmaxineq}) for the vertex $v$, that is,
        $$
            \reg(S/J_G) \leq \max\{ \reg(S_{v}/J_{G-v}), \reg(S/J_{G_{v}}), \reg(S_{v}/J_{G_{v}-v})+1 \}.
        $$
    Clearly, $\ell(G-{v})=2$ and $c(G-{v})=2$. By Theorem \ref{lowerupperbreg} and Theorem \ref{numberofmc}, we find $\reg(S_{v}/J_{G-v})=2$. Similarly, $\ell(G_{v})=2$ and $c(G_{v})=2$ imply that $\reg(S/J_{G_v})=2$. Also, $G_v-v$ is a complete graph and $\reg(S_v/J_{G_v-v})=1$ by Theorem \ref{completereg}. 
    Thus, we have $\reg(S/J_G) \leq 2$. Since $\ell = 2$, Theorem \ref{lowerupperbreg} implies that $\reg(S/J_G)=2$. If $u$ is adjacent to all of the vertices $1,\ldots,t$, then $u$ is not adjacent to at least one of $t+1,\ldots,\omega$ due to $\overline{\omega}=3$. Applying the inequality (\ref{regmaxineq}) consecutively to the vertices $1,\ldots,t$ and using Theorem \ref{completereg}, we see $\reg(S/J_G) \leq \max \{ \reg(S'/J_{G'}), 2\}$ where $G'=G-\{1,\ldots,t\}$ and $S'$ is the polynomial ring obtained from $S$ by removing the variables indexed by $1,\ldots,t$. Here, we have $\reg(S'/J_{G'}) = \reg(S'_v/J_{G'-v})$ since $v$ does not have any neighbours in $G'$. Also, $c(G'-v)=2$ imply that $\reg(S'_v/J_{G'-v}) \leq 2$. Hence, we obtain $\reg(S/J_G)=2$.

    \vspace{3mm}
    \noindent \textit{Case 2:} Suppose that $\{v,u\} \notin E(G)$. Then, $v$ is adjacent to some vertices in $F$, say $1,\ldots,t$, where $t < \omega$ due to $\overline{\omega}=3$. Similarly, $u$ is adjacent to some vertices in $F$. If $u$ has a neighbour $i \in F$ such that $i > t$, then $u$ is adjacent to all vertices $1,\ldots,t$ due to $\ell (G) =2$. Applying the inequality (\ref{regmaxineq}) consecutively to the vertices $1,\ldots,t$ and using Theorem \ref{completereg}, we see $\reg(S/J_G) \leq \max \{ \reg(S'/J_{G'}), 2\}$ where $G'=G-\{1,\ldots,t\}$ and $S'$ is the polynomial ring defined in Case 1. Then, it can be seen that $\reg(S'/J_{G'}) = \reg(S'_v/J_{G'-v})$ since $v$ does not have any neighbours in $G'$. Also, we see $\reg(S'_v/J_{G'-v}) \leq c(G'-v)=2$. Hence, we obtain $\reg(S/J_G)=2$. If $u$ is not adjacent to any $i \in F$ with $i>t$, then the neighbours of $u$ are $i_1,\ldots,i_s$ where $\{i_1,\ldots,i_s\} \subseteq \{1,\ldots,t\}$. By applying the inequality (\ref{regmaxineq}) consecutively to the vertices $i_1,\ldots,i_s$, we see that $\reg(S/J_G) = 2$ by similar arguments.
    
    So, we conclude that there is no connected graph for the values $\ell=2$, $r=3$ and $\overline{\omega}=3$.
\end{proof}
 This leads to the following natural question, whose answer remains open.
 
    \begin{qquestion}
        For which values of $r$ and $\overline{\omega}$ such that $2 \leq r \leq \overline{\omega}$, does there exist a connected graph $G$ satisfying $2=\ell(G)$, $r=\reg(S/J_G)$ and $\overline{\omega} = |V(G)| - \omega(G) +1$?
    \end{qquestion}

	\bigskip

    \section*{Acknowledgments}
 The authors thank the anonymous referees for their valuable suggestions and corrections. We especially thank them for raising the question that prompted our investigation of Conjecture \ref{regpdimconj} and led to the counterexample presented in Example \ref{ContraExample}. The AI assistant Claude (Anthropic) was used to explore potential counterexamples. It identified the graph in Example \ref{ContraExample} in the class of WL-graphs, which was independently verified by the authors using Macaulay2 \cite{Mac2}. The authors take full responsibility for the final content of this paper.

    \bigskip
	\noindent \textbf{Funding} No funding was received to assist with the preparation of this manuscript.
	
	\bigskip
	\noindent \textbf{Data Availability} Data sharing is not applicable to this article as no datasets were generated or analyzed during the current study.
	
	\section*{Declarations}
	
	\noindent \textbf{Conflict of interest} The authors declare that they have no conflict of interest.

\end{document}